\documentclass{amsart}
\usepackage{latexsym,amssymb,amsmath,amsfonts,amsthm, enumerate, tikz, color,float}
\usepackage{comment}

\newtheorem{thm}{Theorem}

\newtheorem{lemma}[thm]{Lemma}

\newtheorem{cor}[thm]{Corollary}
 
\def\beq{ \begin{equation} }
\def\eeq{ \end{equation} }

\def\square{\vcenter{\vbox{\hrule height .4pt
  \hbox{\vrule width .4pt height 5pt \kern 5pt
        \vrule width .4pt} \hrule height .4pt}}}

\newcommand{\f}{\frac}
\newcommand{\OZ}{\vec{\mathbb Z}}
\DeclareMathOperator{\Bin}{Bin}

\newcommand{\0}{\mathbf 0}

\title{Coexistence in chase-escape}
\author{Rick Durrett, Matthew Junge, and Si Tang}

\usepackage{graphicx}
\graphicspath{%
    {converted_graphics/}
    {/}
}

\begin{document}

\maketitle  
 
\begin{abstract}
We study a competitive stochastic growth model called chase-escape in which red particles spread to adjacent uncolored sites and blue only to adjacent red sites. Red particles are killed when blue occupies the same site. If blue has rate-1 passage times and red rate-$\lambda$, a phase transition occurs for the probability red escapes to infinity on $\mathbb Z^d$, $d$-ary trees, and the ladder graph $\mathbb Z \times \{0,1\}$. The result on the tree was known, but we provide a new, simpler calculation of the critical value, and observe that it is a lower bound for a variety of graphs. We conclude by showing that red can be stochastically slower than blue, but still escape with positive probability for large enough $d$ on oriented $\mathbb Z^d$ with passage times that resemble Bernoulli bond percolation. 
\end{abstract}

\section{Introduction}

\emph{First-passage percolation} is a stochastic growth model that can be used to describe the spread of species. The classic first passage percolation model on a connected graph $G=(V, E)$ assumes that each edge $e$ is  assigned a random variable $t_{e}$ independently drawn from a common nonnegative distribution $F$, and these random variables represent the passage time needed to traverse the edges.  Specifically, for any edge $e = (x, y) \in E$, if at time $t_{0}$, a particle occupies site $x$, then after time $t_{x,y}$, the particle at $x$ places a copy of itself at $y$, provided $y$ is empty at time $t_{0}+t_{x,y}^{-}$.
  
For any two vertices $x, y \in V$,  define $T(x, y) = \inf_{\gamma: x\to y} \sum_{e\in \gamma} t_{e}$
to be the shortest time needed for a particle at $x$ to reach $y$. When $G=\mathbb Z^{d}$, the ball $D_{t} := \{\mathbf x \in \mathbb R^{d}: T(\mathbf 0, \lfloor \mathbf x \rfloor) \le t\}$ has a limit shape \cite{cox1981}, where $\lfloor \mathbf x \rfloor \in \mathbb Z^{d}$ is the nearest lower-left lattice point for $\mathbf x\in \mathbb R$. Specifically, under a mild moment condition on $F$,  there exists a convex, axis-symmetric set $\mathcal D_{F}\subseteq \mathbb R^{d}$ such that for any $\epsilon >0$,
\begin{align}
	P\left((1-\epsilon) \mathcal D_F \subseteq \f{D_t}t \subseteq (1+\epsilon)\mathcal D_{F}, \text{ for all $t$ sufficiently large} \right) = 1.\label{eq:shape}
\end{align}
Understanding $\mathcal D_F$ and the boundary fluctuations of $D_{t}$ remains a challenging area of research \cite{fpp}. Even the Markov setting in which the passage times are unit exponential random variables ($F(x) = 1 - e^{-x}$) known as the \emph{Richardson growth model} is far from being satisfactorily understood \cite{richardson}.

There are various extensions of the first-passage percolation model to allow two types of particles (say, red and blue) to spread and interact \cite{bernoulli1,bernoulli2, 2typesZ, pemantle, hoffman}. In these extensions, each edge $e$ is assigned two passage times $(t_{e}^{R}, t_{e}^{B})$, one for red and one for blue. A rule is imposed to define the interaction of the two types. For example, we 
may declare that the first color to reach a site occupies it for all time and blocks the other type from expanding through the site. In this case, if one color is surrounded by the other, then it cannot grow any further. Coexistence in this setting means that red and blue particles can both reach infinity. H\"aggstr\"om and Pemantle \cite{pemantle} demonstrated that this occurs with positive probability on $\mathbb Z^2$ for the special case where  $t_{e}^R$ and $t_{e}^{B}$ are both exponential random variables with equal rates. Hoffman \cite{hoffman} extended this to the four-color setting. Coexistence is especially relevant here, because it implies the existence of multiple geodesics in ordinary first passage percolation \cite{fpp}.
 Note that it is conjectured that coexistence cannot occur if red and blue rates are different \cite[Section 6]{fpp}. 

We consider a two-type spatial growth process called {\em chase-escape}. 
In this process, red particles only spread over the empty sites using the passage times $t_{e}^{R}$, and blue particles take over the red sites using the passage time $t_{e}^{B}$. More precisely, for any edge $(x, y)\in E$, (i) if there is some $t_{0}\ge 0$ so that site $x$ is red and site $y$ is empty during time period $[t_{0}, t_{0}+t_{x,y}^{R})$, then $y$ turns red at time $t_{0}+t_{x, y}^{R}$; (ii) if there is some $t_{0}\ge 0$ so that site $x$ is blue and site $y$ is red during time period $[t_{0}, t_{0}+t_{x,y}^{B})$, then $y$ turns blue at time $t_{0}+t_{x, y}^{B}$.  

%
We learned of this process from Lalley. It is inspired by the spread of two species---host and parasite---through an environment. For example, brush may spread to neighboring empty patches of soil while an infection transmits among the roots. It is interesting to ask how the environment (graph) and spreading rates affect the coexistence of both species. A closely related variant of chase-escape was studied by Kordzakhia on trees \cite{tree1}. Later, Kortchemski considered the process on trees and the complete graph \cite{tree_chase, complete}. Chase-escape can also be viewed as scotching a rumor. This interpretation was studied by Bordenave in \cite{rumor}. The continuous limit of rumor scotching is known as the assassination process and was considered many years earlier by Aldous and Krebs \cite{aldous}.


Unless stated otherwise, we always assume that $t_{e}^R$ is sampled from a rate-$\lambda$ exponential distribution (denoted $\text{Exp}(\lambda)$), with  $t_{e}^B$ from a rate-$1$ exponential distribution.  We let $P_\lambda( \cdot)$ be the probability measure on chase-escape with these passage times. Let $A= A(G)$ be the event that red and blue particles coexist at all times on the graph $G$. For simplicity, we always assume that $G$ is rooted at a vertex $\rho$ and consider the initial configuration, where a red particle is at $\rho$ and a blue particle is at an auxiliary vertex $\mathfrak b \notin G$ attached to $\rho$ by an edge $\mathfrak e$.  We say that \emph{coexistence} occurs if $P_\lambda(A) >0$. Figure \ref{fig:cande1} shows three realizations of chase-escape on $\mathbb Z^2$. On a given graph, it is natural to ask how $\lambda$ affects coexistence. Define 
$$\lambda_c(G):= \sup \{ \lambda \colon P_\lambda(A) = 0\}$$
to be the fastest red expansion rate for which coexistence does not occur. Note that it is not obvious how to prove that $P_\lambda(A)$ is monotonic in $\lambda$, thus it is not clear if $\lambda_c(G)$ is also equal to $\inf \{ \lambda \colon P_\lambda(A) >0\}$. This is discussed in more detail at the end of this section.

As mentioned earlier, a similar model called \emph{escape} was studied first by Kordzakhia on homogeneous trees \cite{tree1}. In the escape model, blue can also spread to empty sites. Note that when the underlying graph is a tree, the survival of red is equivalent in either escape or chase-escape. Kordzakhia gave an explicit formula for $\lambda_c(\mathbb T_d)$, but did not work out whether survival occurs when $\lambda= \lambda_c(\mathbb T_d)$. This was answered later by Bordenave \cite{bordenave2014extinction} in the more general setting of Galton-Watson trees. He proved that red does not survive on any Galton-Watson tree with mean degree $d$ at criticality. This includes the special case of a $d$-ary tree. Kortchemski also studied the process on trees and made some progress at describing  the number of surviving particles at each level \cite{tree_chase}. We provide a much simpler calculation of $\lambda_c(\mathbb T_d)$ in Theorem \ref{thm:tree} that includes the behavior at criticality. Our main interest, however, is in the process on graphs with cycles.

The shape theorem at \eqref{eq:shape} implies that for $G=\mathbb Z^{d}$ and $\rho =\mathbf 0:= (0, \ldots, 0)$ 
\begin{align}\lambda_c(\mathbb Z^d) \leq 1,  \text{ for all $d \geq 1$} .\label{eq:1}
\end{align}
To see this, suppose $\lambda > 1$ and let $\mathcal D_{1}$ be the limit shape of the one-type Richardson growth model. If red and blue were to grow over the empty sites on two $\mathbb Z^{d}$ lattices separately, starting from $\mathbf 0$ and at rates $\lambda$ and $1$, respectively, then, for $\epsilon = (\lambda-1)/3$,  there exists some $T$ such that
\begin{equation}
\label{eqn:zdcriticallambda}
 B(t)\subseteq (1+\epsilon)t\mathcal D_{1} \subsetneq(1+2\epsilon)t\mathcal D_{1} \subseteq R(t), \ \text{for all } t > T,
\end{equation}
where $R(t)$ (resp.\ $B(t)$) denote the sites that can be reached by red (resp.\ blue) particles by time $t$. Recall that in the chase-escape model, we always assume that at time $t=0$, a blue particle is attached to the red particle at $\mathbf 0$ by a special edge $\mathfrak e$. Thus, red survives whenever $t_{\mathfrak e}^{B} > T$, where $T$ is the special time such that \eqref{eqn:zdcriticallambda} holds, and the event $\{t_{\mathfrak e}^{B}>T\}$ occurs with positive probability.

\begin{figure}
\begin{center}
\begin{center}
\includegraphics[width = 3.5 cm]{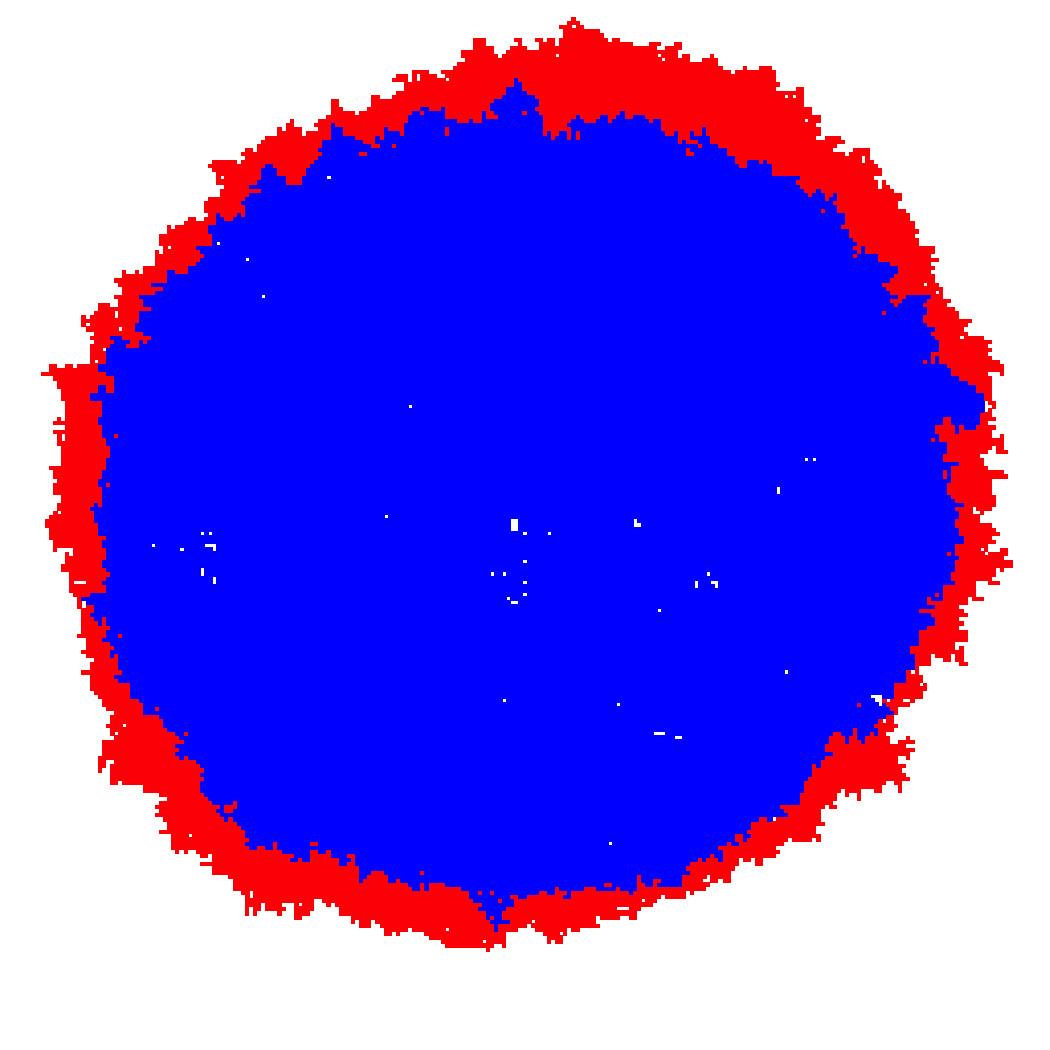} \qquad 
\includegraphics[width = 3.5 cm]{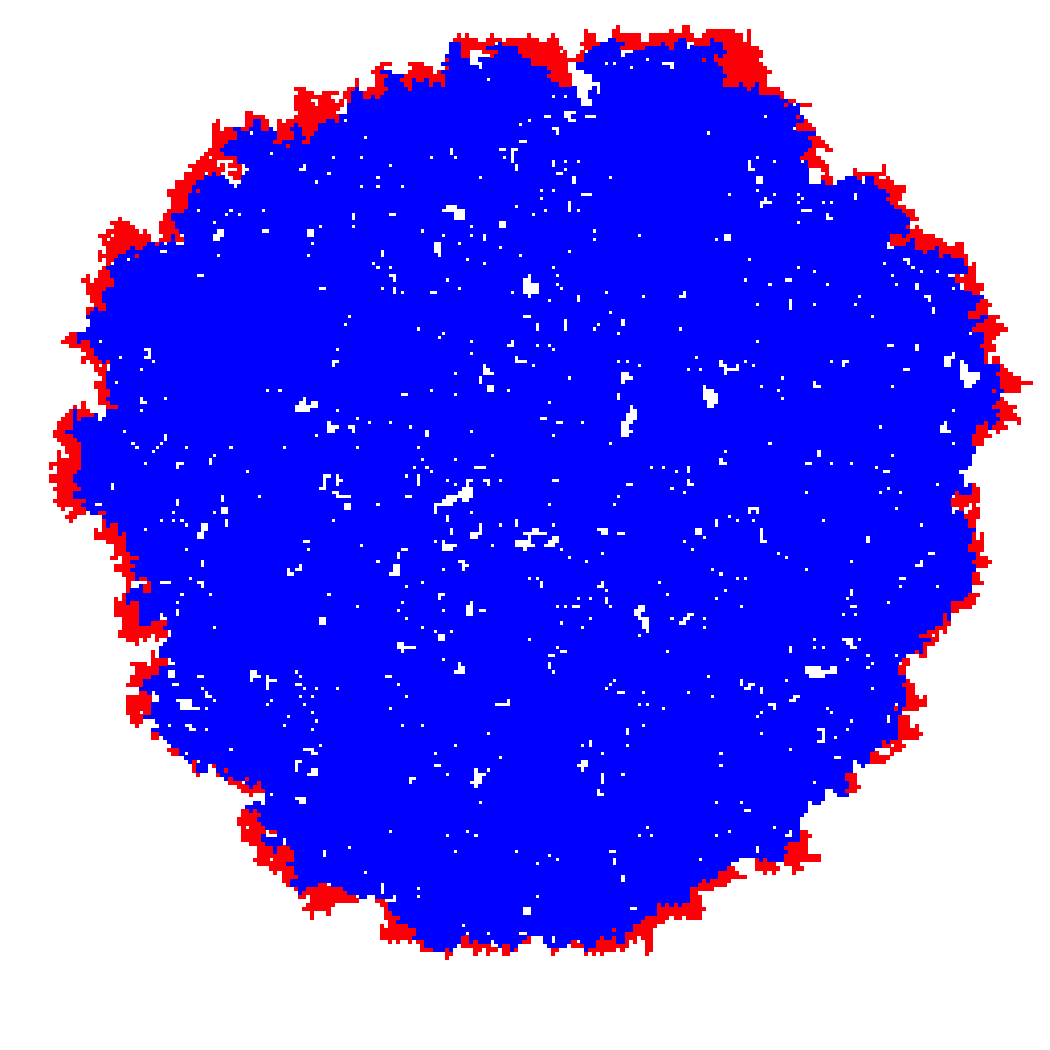} 
\includegraphics[width = 3.5cm]{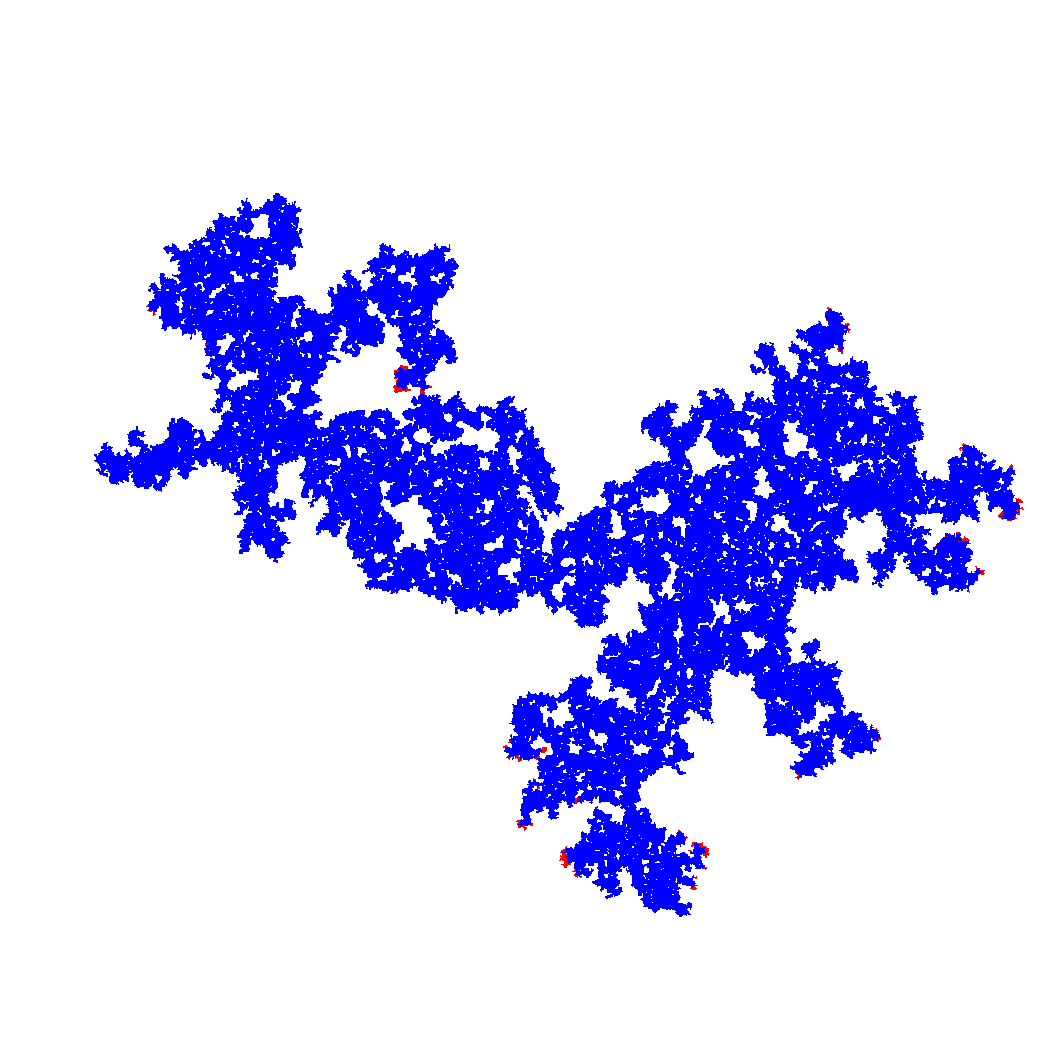}
\end{center}
\end{center}
\caption{Chase-escape on $\mathbb Z^2$ with $\lambda =1$ (left), $\lambda = 0.75$ (middle) and $\lambda =0.5$ (right). Simulations suggest that red can survive despite being slower.} \label{fig:cande1}
\end{figure}

By comparing to a nearest-neighbor random walk, it is straightforward to prove that $\lambda_c(\mathbb Z) =1$ (see Lemma \ref{lem:Z}). The discussion at the end of \cite{complete} credits James Martin with the  conjecture that $\lambda_c(\mathbb Z^2) <1$. Simulations in \cite{si} suggest that red can survive for $\lambda <1$ and that $\lambda_c(\mathbb Z^2) =1/2$ (see Figures \ref{fig:cande1} and \ref{fig:crit}).  
 This is surprising, because, if true, red can still survive even when it spreads significantly slower than blue. One ``advantage" red has is that if blue reaches the red boundary it gets slowed down. Another is that there are regions that red spreads unusually fast across. Since red and blue passage times are independent, blue is unlikely to also spread quickly. Using these advantages to form a rigorous proof seems closely related to understanding the boundary fluctuation of the Richardson growth model, and geodesics in first passage percolation. Both of these objects are notoriously difficult to describe \cite{fpp}. 
 
 Proving coexistence in $\mathbb Z^d$ for some pair $(d,\lambda)$ with $d \geq 2$ and $\lambda \leq 1$ would be very interesting. We further conjecture that $\lambda_c(\mathbb Z^d) \downarrow 0$ as $d \uparrow \infty$. This may be hard to prove since it is unclear that the process is monotonic in $\lambda$ or $d$. Intuitively, $P_\lambda(A)$ ought to increase with $\lambda$ and decrease with $d$, and this is indeed the case on trees. However, there is no obvious coupling that establishes this on graphs with cycles. The issue is that when red spreads faster, then blue also speeds up. It is not even obvious that there is a single phase transition. 
 
 \begin{figure}
\begin{center}
\begin{center}
\includegraphics[width = 6cm]{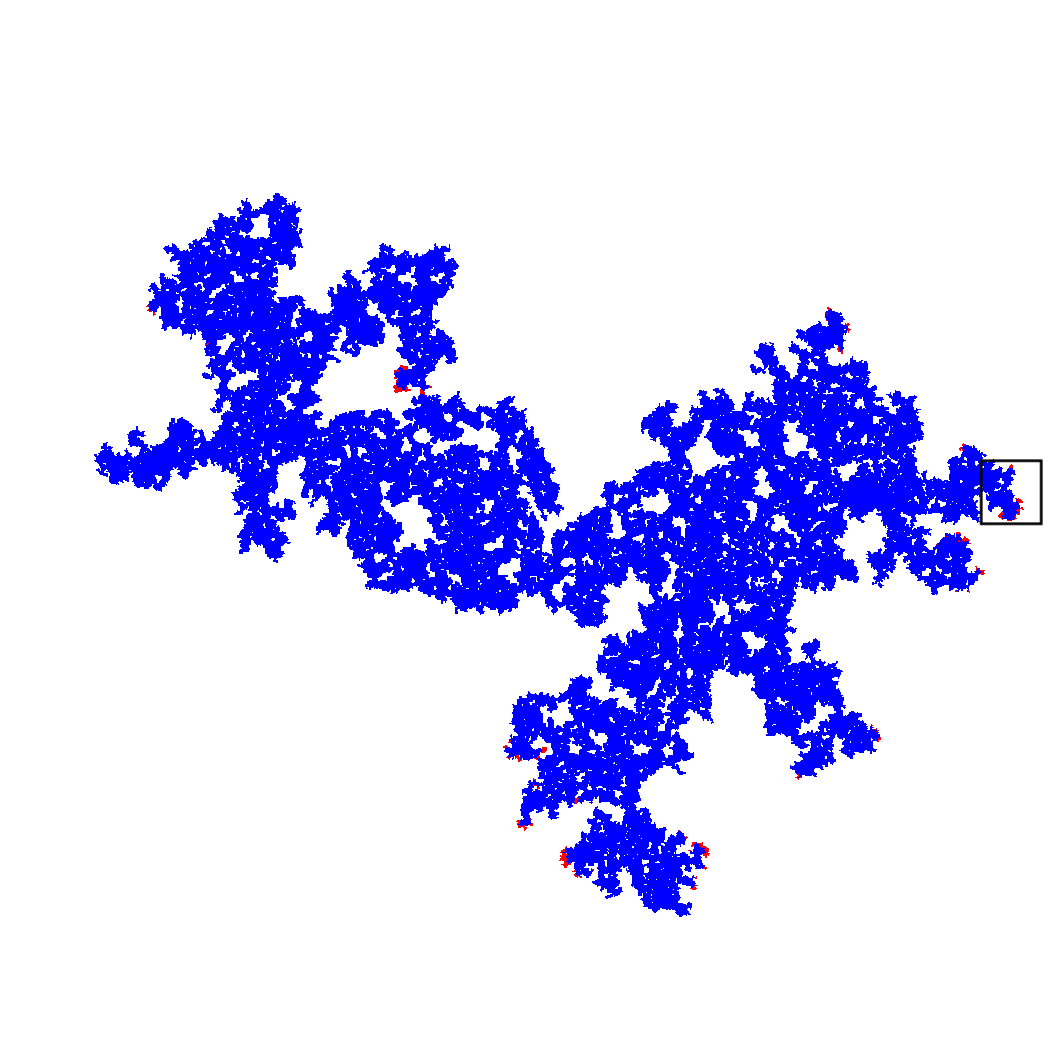} \qquad 
\includegraphics[width = 4cm]{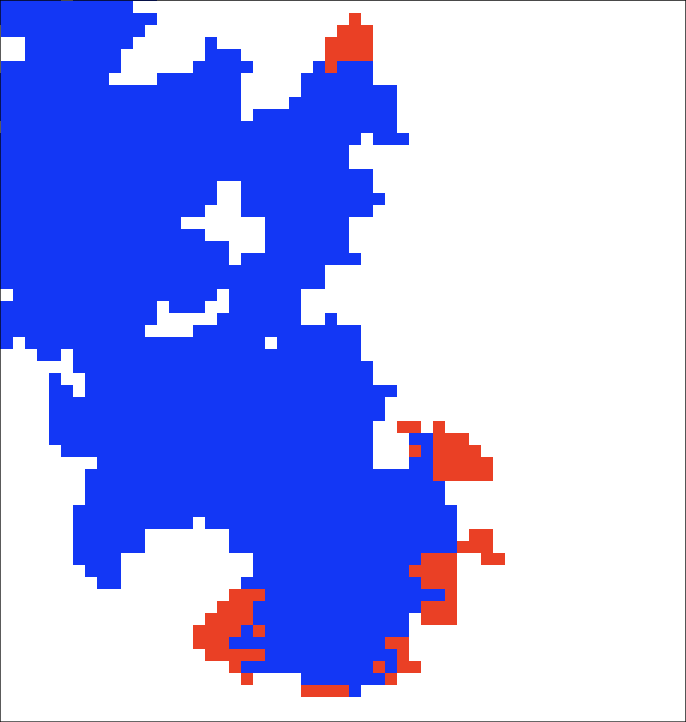}
\end{center}
\end{center}
\caption{Chase-escape on $\mathbb Z^2$ at the conjectured critical value $\lambda = 0.5$ (left). The right image shows that there are still surviving red buds at the boundary of the fractal-like blue region.} \label{fig:crit}
\end{figure}
 
 \subsection{Results}

 Establishing a coexistence phase where red is slower than blue is tractable on the infinite $d$-ary tree $\mathbb T_d$ (a rooted tree where each vertex has $d$-children). Our first result is a new proof that computes $\lambda_c(\mathbb T_{d})$ and determines the behavior at criticality. Note, as discussed earlier, this result was already proved in \cite{tree1, bordenave2014extinction}. Our proof is different though and much shorter.

\begin{thm} \label{thm:tree}
On the $d$-ary tree with $d \geq 2$, $$\lambda_c(\mathbb T_d) =2d -1 - 2 \sqrt { d^2 - d} \sim \f 1 {4d},$$ and $P_{\lambda_c}(A) =0$.
\end{thm}
The proof that $P_\lambda(A) = 0$ for $\lambda \leq \lambda_c$ uses the fact that, restricted to a ray on $\mathbb T_d$, the distance between red and blue is a nearest neighbor random walk on the nonnegative integers. We can use the setup from Lemma \ref{lem:Z} and the asymptotic behavior for the return time to 0 of a random walk to show that the expected number of sites at distance $n$ to be colored red is summable. This implies coexistence does not occur. To prove $P_\lambda(A) >0$ for $\lambda > \lambda_c$ we embed a Galton-Watson process that describes the number of red particles that reach generation $kN$, $k \in \mathbb N$. When $N$ is large enough, the Galton-Watson process is supercritical, and thus $P_\lambda(A) >0$. 

Our main interest is studying chase-escape on graphs with cycles. The argument that proves $P_\lambda(A) = 0$ for $\lambda \leq \lambda_c(\mathbb T_d)$ can be easily extended to any graph on which the number of length-$n$ paths grows no faster than exponential. 

\begin{cor} \label{cor:paths}
	Fix a graph $G$ and let $\Gamma_n$ be the number of self-avoiding length-$n$ paths starting from the root. If $|\Gamma_n| \leq d^n$ for all $n \geq 1$, then $\lambda_c(\mathbb T_d) \leq \lambda_c(G)$.
\end{cor}

We initially thought it possible that $\lambda_{c}(\mathbb Z \times \{0,1\})$ was strictly less than 1. This was because blue is slowed down when it reaches the not-entirely-filled red boundary. However, this effect is negligible, because the mixed region of red and uncolored sites is exponentially unlikely to be large. 

\begin{thm} \label{thm:ladder} For $G=\mathbb Z \times \{0,1\}$ with the subgraph structure induced by $\mathbb Z^2$ and $\rho=\0$,
$\lambda_c(G) =1$
\end{thm}

To prove $\lambda_{c}(\mathbb Z \times \{0,1\}) \le 1$, we show that for any $\lambda < 1$, there is $K < \infty$ such that blue is within a distance $K$ of the red boundary infinity often. Whenever this occurs, there is a positive probability that blue overtakes all of the surviving red sites.  Thus, red is caught eventually. While we did not work out the details, this result suggests that red is caught on any strip of a fixed height. Additionally it ought to hold that $P_1(A)=0$ on the ladder, however, our proof relies heavily on the fact that blue moves strictly faster than red and some new ideas are required.


%


Next, we show that it is possible on the lattice for red to spread slower than the blue, but still have a chance to survive.  
We consider the chase-escape model on the oriented lattice $\OZ^d$, in which particles can only occupy their neighboring sites along the oriented edges $\{\vec e_{x, x+b_{i}}\}_{x\in \mathbb Z^{d}, i=1,\ldots, d}$, with $\{b_1,\hdots, b_d\}$ being the standard, positive basis vectors for $\mathbb Z^{d}$. Initially, $\mathbf 0 = (0,\ldots, 0)$ is red and $(-1,0, \ldots, 0)$ is blue. Red and blue passage times for each directed edge $\vec e$ are sampled independently from the two distributions 
\begin{align} 
t_{\vec e}^{R}&\sim p\delta_1 +  (1-p) \delta_\infty \nonumber \\
t_{\vec e}^{B}&\sim p\delta_0 + (1-p) \delta_m. \label{eq:percB}
\end{align}
With atomic passage times there could be moments that blue arrives at a red site at the exact time that red would spread to the next site. To make the process well-defined, we invoke the rule that blue catches red and prevents it from spreading if blue arrives at or before the time red would move.  Similar competitive growth models that mimic Bernoulli percolation are studied by Garet and Marchand in \cite{bernoulli1, bernoulli2}. The competition dynamics they consider are not chase-escape, but rather more like the dynamics in \cite{pemantle}.

Since edges with $\infty$ passage times are never used, the collection of all edges for which red passage time equals one is equivalent to the open clusters in \emph{bond percolation}, the random subgraph obtained when each edge in $\mathbb Z^{d}$ is independently retained (open) with probability $p$ and removed (closed) with probability $1-p$. Many aspects of bond percolation are well understood. Letting $\mathcal C$ be the connected component containing the origin, it is known that there is a critical value $\vec p_c(d)$ and if $p>\vec p_{c}(d)$ we have $P_p(|\mathcal C| = \infty) >0$ \cite{durrett_oriented}. 
 

Red particles never use edges with $\infty$ passage times to expand. However, blue can still ``jump" across such  an edge if the other endpoint was colored red along some other path with no infinite red passage time. Uninhibited, blue would spread much faster than red. However, the presence of sufficiently many ``dead ends"---places where red has $\infty$ passage times in all directions---cuts blue off from spreading and makes it possible for red to escape for small $p$ and large $d$ and $m$. Let $A= A(d,p,m)$ be the event that there are always surviving red particles in the percolation-like chase-escape model on $\OZ^d$. We show that if $d$ and $p$ are chosen appropriately, and $m$ is made large then red can survive. 
\begin{thm} \label{thm:site_perc}  For  $G = \OZ^{d}$ and $\rho = \mathbf 0$, there is a choice of $d$  and $p > \vec p_c(d)$ such that 
$$\lim_{m \to \infty} P_{p,m}(A) = (1-p) P_p(|\mathcal C|= \infty).$$
In particular, $P_{p,m}(A) >0$ for some choice of $d,p,$ and $m$. 
\end{thm}

It is slightly unsatisfying to require that $d$ be large. The reasoning in the proof of Corollary \ref{cor:paths} can also be used to show that for $p>\vec p_c(2)$, and small values of $m$, that $P_{p,m}(A(\OZ^2))=0$. Thus, if  
\begin{align}\text{$P_{p_0,m_0}(A(\OZ ^2))>0$ for some $p_0$ and $m_0$}\label{eq:Q},
\end{align}	
then there is a phase transition in $P_{p_0,m}(A(\OZ ^2))$ as we increase $m$.  As before there is no obvious monotonicity in $m$, so we cannot rule out the possibility of multiple phase transitions.

\section{The path, tree, and ladder}

\subsection{The path}

\begin{lemma} \label{lem:Z}
$\lambda_c(\mathbb Z) =1$ and $P_1(A) = 0$.
\end{lemma}

\begin{proof}
Since the process evolves independently in the positive and negative directions, it suffices to prove that red survives on $\{-1,0,1,2,\hdots\}$ with a blue particle at $-1$ and a red particle at $0$ initially. Let $R_{t}$ be the number of sites ever occupied by the red particles up to time $t$ and let $B_{t}$ be the corresponding quantity for blue particles. 
Define $\tau_{n}=\inf\{t\colon R_{t}+B_{t}=n+2\}$, and $\tau_{n}=\infty$ if there is no such $t$. Let $D_t$ be the distance between the rightmost red and blue sites at time $t$ and $S_ n  = D_{\tau_n}$. Notice that $(S_n)$ is a nearest-neighbor random walk, starting at $S_0=1$, where $0$ is an absorbing state. Due to the independence of red and blue passage times, we have
\begin{align}
p:=P_\lambda(S_{n+1} = S_n +1) = P( \text{Exp}(\lambda) \leq \text{Exp}(1)) = \f{\lambda}{\lambda +1}.\label{eq:p}
\end{align}
When $\lambda \leq 1$, the extinction of red is equivalent to the above $p$-biased random walk visiting zero. This is well known to be a.s.\ finite. 
\end{proof}

\subsection{The tree}
To study the process on the tree we need an asymptotically precise estimate for the probability that red can reach a site at distance $n$ in the $G=\mathbb Z$ case. We prove it here. 

\begin{lemma} \label{lem:Zn}
Let $A_n=A_n(\lambda)$ be the event that site $n$ is ever colored red in the chase-escape model on $\mathbb Z$, and set $a_n= \frac{[4p(1-p)]^n }{n^{3/2}}$, where $p = \lambda /(\lambda+1)$. 
\begin{enumerate}[\textup(i\textup )]
\item For some $c >0$ and all $n\geq1$, $P_\lambda(A_{n}) \geq c\, a_{n}$
\item If $\lambda <1$, then for some $C_{\lambda} >0$ and all $n\geq1$, $P_\lambda(A_{n}) \leq C_\lambda\, a_{n}$;
\end{enumerate}
\end{lemma}

\begin{proof}
Let $(S_k)$ be the nearest-neighbor $p$-biased random walk as in \eqref{eq:p}. Note that the event $A_{n}$ is equivalent to the event that $S_{k}$ remains strictly positive for the first $2n$ steps, i.e., 
\[
P_\lambda(A_{n}) =P_\lambda(S_{k}\ge 1, k\le 2n|S_{0}=1) = \sum_{a=0}^{n}P_\lambda(S_{k}\ge 1, k\le 2n; S_{2n}=2a+1|S_{0}=1). 
\] 
Since $S_{2n}$ must have the same parity as $S_{0}$, we only considered the cases where $S_{2n}$ is odd. For any random walk path of length $2n$ from $1$ to $2a+1$ that does not hit zero, there must be $(n+a)$ steps to the right and $(n-a)$ steps to the left. The total number of such paths is 
\[
\binom{2n}{n+a}-\binom{2n}{n+a+1},
\] 
following a simple reflection principle. We then have
\begin{align}
\notag
P_\lambda(A_{n}) &= p^{2n}+ \sum_{a=0}^{n-1}\left[\binom{2n}{n+a}-\binom{2n}{n+a+1}\right] p^{n+a} (1-p)^{n-a}\\
\label{eqn:pan-est}
&=p^{2n}+p^{n} (1-p)^{n} \sum_{a=0}^{n-1}\binom{2n}{n+a} \frac{2a+1}{n+a+1}  \left(\frac{p}{1-p}\right)^{a}\\
\notag
&\le p^{2n} + \frac{p^{n}(1-p)^{n}}{n+1}\binom{2n}{n}\sum_{a=0}^{\infty}(2a+1)  \left(\frac{p}{1-p}\right)^{a}.
\end{align}
Since $\lambda < 1$, we have $p <\frac{1}{2} <1-p$, the summation above is finite. Moreover, 
$\frac{1}{n+1}\binom{2n}{n}$ is the $n$-th  Catalan number, which is known to be of order $\frac{4^n}{n^{3/2}}$ for large $n$. Putting $C_{\lambda} = \sum_{a=0}^{\infty}(2a+1)  \left(\frac{p}{1-p}\right)^{a}$, we have the desired upper bound. The lower bound is obtained by looking at the $a=0$ term in \eqref{eqn:pan-est}  and using the asymptotic behavior of Catalan numbers.
\end{proof}

\begin{proof}[Proof of Theorem \ref{thm:tree}] 

We consider the initial configuration in which the root is red and a special vertex $\mathfrak b$ attached to the root is blue. First note that if $\lambda >1$, then the distance between red and blue along an arbitrary path to $\infty$ is equivalent to chase escape on $\mathbb Z$. By Lemma \ref{lem:Z} we have $P_\lambda(A) >0$ in this case.

 Now, suppose that $\lambda \leq 1$.  Let $R_n$ be the number of sites at distance $n$ that are ever colored red and $R= \sum_{n=1}^{\infty} R_n$ be the total number of red sites. Notice that red survives a.s.\ if and only if it occupies infinitely many sites. Thus, $P_\lambda(A) = P_\lambda( R=\infty)$. We show that $P_\lambda(R=\infty) = 0$ for $\lambda$ small enough by proving $E R <\infty.$

For any vertex $v \in \mathbb T_d$, let $|v|$ denote its graph distance from the root. Let $A(v)$ be the event that $v$ is ever colored red. Since the tree has no cycles, we have $P_\lambda(A(v)) = P_\lambda(A_n)$ for any $v \in \mathbb T_d$ with $|v| = n$, with $A_n$ the event that red reaches a distance $n$ on a fixed path as in Lemma \ref{lem:Zn}. 
Linearity of expectation and the bound from Lemma \ref{lem:Zn} gives
\begin{align}
E R_n &= E\sum_{|v| = n} \mathbf 1_{A(v)} = \sum_{|v| = n} P_\lambda( A(v) ) \leq C_{\lambda}\frac{ d^n (4p(1-p))^n}{n^{3/2}}. \label{eq:S} 
\end{align}
Observe that $\lambda_c(d)$ is the smallest solution of 
$$4p(1-p)d =\f{4 d \lambda}{(1+ \lambda)^2} = 1.$$
It is straightforward to verify that $4dp(1-p)\leq 1$ for $\lambda \leq \lambda_c(d)$, and in this case $E R_n$ is summable, and thus $ER < \infty$. 

To prove that $P_\lambda(A) >0$ for $\lambda > \lambda_c(d)$, observe that the lower bound in Lemma \ref{lem:Zn} ensures that for some fixed, large $N$ we have $d^NP_\lambda(A_N) >1.$ Thus, the expected number of sites at distance $N$  that are ever colored red is strictly greater than $1$. When first occupied by red, the distance from each of these sites to the nearest blue particle is at least one. Since the tree has no cycles, the survival probability of chase-escape is monotonic on a tree; both respect to $\lambda$ and the initial distance blue starts from red. This means that moving the chasing blue particles to distance 1 from each of red site at distance $N$ will result in fewer surviving red particles at distance $2N$. Thus, the number of sites colored red at distances $N,2N,\hdots$ dominates  a Galton-Watson process with mean $d^N P_\lambda(A_N) >1$. This is supercritical, and thus $P_\lambda(R= \infty)>0$. 
\end{proof}

We can quickly deduce that the critical speed for any graph with less than $d^n$ self-avoiding paths of length $n$ is no smaller that of a $d$-ary tree. 

\begin{proof}[Proof of Corollary \ref{cor:paths}] 
For each $v \in G$, again let $A(v)$ be the event that $v$ is ever colored red. If $A(v)$ occurs, then along one self-avoiding path $\gamma\colon \rho \to v$, a red particle must be able to reach a distance $n$ from the root $\rho$ before blue catches it. We denote this event by $A^{\gamma}(v)$ and write
\[
 \mathbf 1_{A(v)} \le \sum_{\gamma:\rho\to v} \mathbf1_{A^{\gamma}(v)}
\]
Define $R_n = \sum_{v: |v| = n} \mathbf 1_{A(v)}$ to be the number of sites at distance $n$ that are ever colored red. Recall that $\Gamma_{n}$ is the collection  of all self-avoiding paths of length $n$ from the root. We have
$$R_n \le \sum_{v: |v|=n}\, \sum_{\gamma:\rho \to v} \mathbf1_{A^{\gamma}(v)} =  \sum_{\gamma \in \Gamma_n} \mathbf1_{A^{\gamma}({\gamma(n)})},$$
where $\gamma(n)$ is the $n$-th vertex along $\gamma$.
It follows from our hypothesis $|\Gamma_n| \leq d^n$  that $E R_n \leq d^n P_\lambda(A_n)$ with $A_n$ as in Lemma \ref{lem:Zn}. As in the proof of Theorem \ref{thm:tree}, if $\lambda \le \lambda_c(\mathbb T_{d})$, then $E R_n$ is summable. This implies that $P_\lambda(A) = 0$. 
\end{proof}

\subsection{The ladder}

\begin{proof}[Proof of Theorem \ref{thm:ladder}]

As with the path, it suffices to consider the one-sided ladder. Note that $\lambda_c(\mathbb Z_{+} \times \{0,1\}) \leq 1$, following a similar shape argument as in the introduction and the shape theorem for first passage percolation on one-dimensional graphs \cite[Theorem 1.1]{ahlberg2015asymptotics}. For the lower bound, we start by considering ordinary first-passage percolation with only one-type, that is, red particles spread at rate $\lambda$ and  blue particles are absent. We will say that empty sites have state $0$ and red sites have state $1$. Let $\xi_{t}(x,y)$ denote the state of site $(x, y)$ at time $t$:
$$
\xi_{t}(x, y) = \left\{
\begin{array}{ll}
0, &\text{if $(x, y)$ is red at time $t$}\\
1, &\text{if $(x, y)$ is empty at time $t$}
\end{array}.
\right.
$$
Let $\eta_t(x) = \xi_t(x,0) + \xi_t(x,1)$. If $\eta_t(x)  =1$ then we say that there is a \emph{hole} at distance $x$ at time $t$. Define $H_t := | \{ x \colon \eta_t(x) = 1 \} |$ to be the number of holes at time $t$.
The process $(H_t)_{t\ge 0}$ is dominated by a time-homogeneous Markov chain with jump rates
\begin{align*}
q(k-1, k) &= 2\lambda \\
q(k,k-1) &=k\lambda
\end{align*}
for all $k\ge 1$. This is because each hole can be filled by any site next to it.
A comparison with an asymmetric random walk, in which $k \to k-1$ happens at rate $3\lambda$ for all $k\ge 3$ and all other transition rates are held the same, shows that the first time to return to 0 (denoted by $T_{0}$) satisfies that for all $n$, $P_1(T_0 \ge n) \le ce^{-\gamma n}$ for some $c, \gamma > 0$.

For any $x, y \in \mathbb Z_{+}\times \{0, 1\}$, let $T(x, y)$ denote the passage time to travel to site $y$ starting from $x$. Define $$\nu := \displaystyle\lim_{n\to\infty} \frac{T(0, \mathcal H_{n})}{n}$$ to be the time constant, where $\mathcal H_{n} = \{(n, 0), (n, 1)\}$. Following an argument similar to \cite[Theorems 5.2 and 5.9]{Aspects} for the $\mathbb Z^{d}$ case, one can show that  for each $\epsilon > 0$, there are constants $A, B>0$ (depending on $\epsilon$) such that for all $n \ge 1$, 
\begin{align}
\label{eqn:large-dev}
&P_\lambda(\,|T(\mathbf 0, \mathcal H_{n}) -n\nu| \ge n \epsilon) < Ae^{-Bn}. 
\end{align}
Let $s_0=0$ and for $k\ge 0$, define
\begin{align*}
		\tau_k &:= \inf\{ t \ge s_k \colon H_t=0 \} ,\\ 
		s_{k+1} &:= \inf\{ t > \tau_k \colon H_t > 0\}.
\end{align*}
At each $\tau_k$, the red region contains no holes, we call this state a \emph{solid red block}, and we call each interval $[\tau_k,\tau_{k+1})$ a \emph{red cycle}.  Note that when blue particles are absent, $\{\tau_{k+1}-\tau_{k}\}_{k\ge 0}$ is stationary. Denote $\sigma:=E(\tau_{k+1}-\tau_{k})$.

Consider now the chase-escape model, where each site can be in one of the three states 0 (empty), 1 (red), and 2 (blue). At each time $t \ge 0$, let $\tilde R_t$ denote the first coordinate of the rightmost red particle and $\tilde B_t$ that of the rightmost blue particle. The definition of a ``hole'' is the same as above, and the sequences, $\{\tilde s_{k}\}_{k\ge 0}$ and $\{\tilde \tau_{k}\}_{k\ge 0}$, and red cycles $\{[\tilde \tau_{k}, \tilde \tau_{k+1})\}_{k\ge 0}$ are defined accordingly in the chase-escape setting.   

Red spreads at rate $\lambda$ with $\lambda <1$ and the blue at rate $1$. Suppose the process is in some configuration at $t=0$. It suffices to show that, independent of this configuration, there are constants $K > 0$ and $\delta > 0$ such that with probability at least $\delta$ at some later time $T < \infty$,  we have $\tilde R_{T}-\tilde B_{T} \le K$ and that $[\tilde B_{T}+1, \tilde R_{T}]\times \{0, 1\}$ is a solid red block.  It follows that such desired configuration will occur infinitely many times with probability one, and if this configuration appears, blue has a positive probability to take over the entire solid red block before it expands further. 

We will make use of a regenerative structure and we will explain the first step in full detail. Note that whenever blue takes over a red site, it never creates new holes, and by the same argument as in the red-only case, one can see that $P_\lambda(\tilde \tau_{0} < \infty)=1$, starting from any initial configuration. If $\tilde R_{\tilde \tau_{0}}-\tilde B_{\tilde \tau_{0}} \le K$, then we are done. Now suppose at $\tilde \tau_{0}$, we have $\tilde R_{\tilde \tau_{0}}-\tilde B_{\tilde \tau_{0}} = M > K$. 
We run the chase-escape model for another
$$N_{1}:=\left\lfloor\frac{M(1-2\epsilon)\lambda \nu}{\sigma}\right\rfloor$$ 
red cycles. Let $R_{t}$  and $\tau_{k}$ be defined as above for the one-type case at rate $\lambda$, and let $B_{t}$ denote the frontier at time $t$ of another independent one-type process of rate-$1$. These two one-type models are coupled with the two-type chase-escape model $(\tilde R_{t}, \tilde B_{t})_{t\ge 0}$ so that (i) passage times for $(R_{t})_{t\ge 0}$ (resp., $(B_{t})_{t\ge 0}$) are the same as those used by the red  (resp., blue) particles in $(\tilde R_{t}, \tilde B_{t})_{t\ge 0}$; (ii) $(B_{t})_{t\ge 0}$ starts from the same configuration as $\tilde B_{\tilde \tau_{0}}$ and (iii) $(R_{t})_{t\ge 0}$ starts from a solid red block that spans from $\mathcal H_{0}$ to $\mathcal H_{\tilde R_{\tilde \tau_{0}}}$ (and hence $\tau_{0}=0$). \\
We observe that
\begin{align}
\label{eqn:indep-red-blue}
&\ P_\lambda(\tilde R_{\tilde \tau_{N_{1}}}-\tilde B_{\tilde \tau_{N_{1}}} \le M\alpha \mid \tilde R_{\tilde \tau_{0}}-\tilde B_{\tilde \tau_{0}} = M  ) \\
\notag
 \ge & \ P_{\lambda}(R_{\tau_{N_{1}}-\tau_{0}}-B_{\tau_{N_{1}}-\tau_{0}} \le M\alpha-M,\ B_{\tau_{N_{1}}-\tau_{0}} < M).
\end{align}
This is because in the right side of the inequality, $\{B_{\tau_{N_{1}}-\tau_{0}} < M\}$ corresponds to the event that in the two independent one-type models, $(R_{t})_{t\ge 0}$ finishes $N_{1}$ red cycles before the blue frontier $B_{t}$ advances for a distance $M$. 
Then, choose any $\epsilon < \frac{1-\lambda}{4}$ and $\alpha := 1-4\lambda \epsilon^{2}\in (0, 1)$ fixed,  we have \eqref{eqn:indep-red-blue} is at least 
\begin{align*}
P_\lambda\left(R_{\tau_{N_{1}}-\tau_{0}} <\lambda M(1-4\epsilon^{2}),\  M(1-4\epsilon) \le B_{\tau_{N_{1}}-\tau_{0}} \le M\right). 
\end{align*}
Since $\tau_{k}-\tau_{k-1}$ has an exponential tail, the total time $(\tau_{N_{1}}-\tau_{0})$ satisfies
\[
P_\lambda(|\tau_{N_{1}}-\tau_{0} - N_{1}\sigma| \le \epsilon N_{1} \sigma) \ge 1-e^{-CN_{1}}, \quad \text{for some } C>0. 
\]
On the event $\{|\tau_{N_{1}}-\tau_{0}-N_{1}\sigma| < \epsilon N_{1}\sigma\}$, 
the blue displacement is within the interval 
$$\left[\frac{N_{1}(1-2\epsilon)\sigma }{\lambda \nu}, \ \ \frac{N_{1}(1+2\epsilon)\sigma }{\lambda\nu}\right] \subseteq \left[M(1-4\epsilon), M\right] $$ 
with probability $\ge 1-C_{2}e^{-C_{3}N_{1}}$. We thus have 
\begin{align*}
&\ \ P_\lambda\left(R_{\tau_{N_{1}}-\tau_{0}} <\lambda M(1-4\epsilon^{2}),\  M(1-4\epsilon) \le B_{\tau_{N}-\tau_{0}} \le M\right)\\
\ge & P_\lambda\left(R_{N_{1}\sigma(1+\epsilon)} <\lambda M(1-4\epsilon^{2}),\  |\tau_{N_{1}}-\tau_{0}-N_{1}\sigma| < \epsilon N_{1}\sigma \right) (1-C_{2}e^{-C_{3}N})
\end{align*}

According to \eqref{eqn:large-dev}, the red frontier extends by a distance at most $N_{1}\sigma(1+2\epsilon)/\nu$ with a probability $\ge 1-Ae^{-BN_{1}}$. 
Combining all these estimates above, we have for some $C_{0}=C_{0}(\epsilon), \theta = \theta(\epsilon)$, with probability at least 
\[ 
(1-C_{2}e^{-C_{3}N_{1}}) (1-Ae^{-BN_{1}}-e^{-CN_{1}}) = 1-C_{0}e^{-\theta M},
\]
the distance between the red and blue frontiers after $N_{1}$ red cycles is less than $M\alpha$, and moreover, at $\tilde \tau_{N_{1}}$, the rightmost blue particle is following a solid red block.  
We iterate this procedure for $s_{M} :=\min\{j\colon M\alpha^{j} \le K\}$ times. For any $M > K$, after $s_{M}$ iterations, the distance between the red and blue frontiers is no more than $K$ with probability at least 
\begin{align*}
\delta_{0} := \inf_{M>K}(1-C_{0}e^{-\theta M})(1-C_{0}e^{-\theta M\alpha})\cdots (1-C_{0}e^{-\theta M\alpha^{s_{M}-1}}) > 0.
\end{align*}
Therefore,  we have
\begin{align*}
& P(\text{at some } T< \infty, \tilde R_{T} -\tilde B_{T} \le K)\\
\ge& \sum_{M=K+1}^{\infty} P(\text{at some } T< \infty, \tilde R_{T} -\tilde B_{T} \le K | \tilde R_{\tilde \tau_{0}} - \tilde B_{\tilde \tau_{0}} = M)P( \tilde R_{\tilde \tau_{0}} - \tilde B_{\tilde \tau_{0}} = M)\\
& \qquad + \ P( \tilde R_{\tilde \tau_{0}} - \tilde B_{\tilde \tau_{0}} \le  K)\vspace*{3mm}\\
\ge &\  \delta_{0}[1-P( \tilde R_{\tilde \tau_{0}} - \tilde B_{\tilde \tau_{0}} \le  K)] + P( \tilde R_{\tilde \tau_{0}} - \tilde B_{\tilde \tau_{0}} \le  K) > \delta >0,
\end{align*}
and the proof is complete.
\end{proof}

\section{Percolation-like chase-escape}
Before getting to the proof of Theorem \ref{thm:site_perc}, we discuss the chase-escape dynamics on the infinite path $[-1,\infty)\cap \mathbb Z$ with $-1$ blue and $0$ red initially. Suppose that all red passage times are $1$ and that blue chases according to \eqref{eq:percB}. Let $t_i^B$ be the passage time for blue along the edge $(i-1,i)$. Define
$$T_n = \sum_{i=0}^n t_i^B.$$ We will suppress the dependence on $p$ and $m$ during the proof and denote the probability measure on this process by $P(\cdot) = P_{p,m}( \cdot)$. 
\begin{lemma} \label{lem:tau} 
Let $B = \{t_0 = m\}$. For any $\epsilon>0$, there exists $M = M(\epsilon)$ such that  for all $n,m>  M$,
\begin{align}
\label{eqn:lemma8}
P(T_n \leq n+1 \mid B) \leq (p+\epsilon)^{n}.
\end{align}
\end{lemma}
\begin{proof} 
If there are more than $(n+1-m)/m$ edges with passage time $m$ then $T_n > n+1$ conditional on $B$. Let $a_{n,m} = \lceil (n+1-m)/m \rceil$ and $X \overset{d} = \Bin(n,1-p)$. It follows from the previous observation that $$P(T_n \leq n+1 \mid B) \leq P( X \leq a_{n,m}).$$
 Taking $m\geq 2$, it follows from a union bound and that $P(X=k)$ is an increasing function for $k \in [0,n/2]$ that 
$$P( X \leq (n+1)/m) \leq a_{n,m} P(X = a_{n,m}) = a_{n,m} \binom{n}{a_{n,m}} p^{n-a_{n,m}}.$$
Stirling's formula implies that given $\epsilon'>0$, there exists a $\delta>0$ such that for all $n> N(\delta)$ it holds that
$$(n+1)\binom{n}{\delta n} \leq (1+ \epsilon ')^n.$$
Let $\epsilon'$ and $\delta$ be small enough so  that $(1+ \epsilon')p^{1-\delta} < (p+ \epsilon).$ Take $M(\delta)$ large enough so that, for $m> M(\delta)$, we have $a_{n,m} \leq \delta n$ for all $n \geq N(\delta)$. Setting $M = \max \{ N(\delta), M(\delta)\}$ gives for $n,m > M$ 
$$P( X \leq a_{n,m}) \leq [(1+ \epsilon') p^{( 1- \delta))}]^n\leq (p+\epsilon)^n.$$ 
Since this is also a bound for $P(T_n \leq n+1\mid B)$, the proof is complete.
\end{proof}

\begin{proof}[Proof of Theorem \ref{thm:site_perc}]The idea is to set up and then apply a union bound that uses the one-dimensional bound in Lemma \ref{lem:tau} and also an observation that bounds the probability of a contiguous path in the cluster of sites accessible by red via $1$-passage times.
 
First we must choose $p$ and $d$ appropriately. Let $p = d^{-1} + d^{-2}$ and $d>3$ be large enough so that $p > p_c(d)$ and also 
\begin{align}1-(1-p)^d < 1- 3d^{-1}.\label{eq:d}\end{align}
 Such a $d$ exists by the main result of \cite{oriented}, which proved $p_c(d) = d^{-1} + d^{-3} + O(d^{-4}).$ 
Furthermore, choose $\epsilon>0$ such that for some $\delta >0$
\begin{align}d (d^{-1} + d^{-2}+\epsilon) ( 1- 3d^{-1}) = 1-\delta\label{eq:epsilon}	.
\end{align}

We describe two events whose occurrence is necessary for red to survive. We must condition on the event $B= \{t^B_{\mathfrak e} =m\}$ i.e., that the first blue edge has passage time $m$. Otherwise, red is instantly caught. Let $\mathcal C$ be the cluster of vertices connected to $0$ via red edges with passage time $1$ and $C = \{ |\mathcal C| = \infty\}$. Our choice of $p$ is larger than $p_c(d)$, thus $P(C)  >0$. Since red can only survive if it reaches infinity in the absence of being chased and $B$ and $C$ are independent, we have
\begin{align}
P(A) =  P(A\mid B,C) P(B)P(C). \label{eq:Rcond}	
\end{align}
As $P(B) = 1-p$, our result will follow from proving that $P(A \mid B,C) \to 1$ as $m \to \infty$. 

Let $\gamma_n$ be an oriented path started at $\0$ with $n$ vertices and edges. Define the event $R(\gamma)$ that every vertex of $\gamma$ belongs to $\mathcal C$. Let $T(\gamma_n)$ be the passage times for blue along the edges of $\gamma_n$. If blue never reaches the red boundary, then red spreads as if it is not being chased and survives on the event $C$. Since red spreads at unit speed along $\mathcal C$, we can express this idea formally as
$$P(A\mid B,C) \geq 1 - P( \cup_{n \geq 1, \gamma_n}  \{T(\gamma_n) \leq n+1\} \cap R(\gamma_n) \mid B,C).$$
Using a union bound and independence of red and blue passage times gives
\begin{align*}P (A \mid B,C) \geq 1 - \sum_{n \geq 1}\sum_{\gamma_n} P(T(\gamma_n) \leq n+1 \mid B) P(R(\gamma_n) \mid C).
\end{align*}
Noting that $P(T(\gamma_n) \leq n +1 \mid B) = 0$ for $n <m$,  we then have
$$P (A \mid B,C) \geq 1 - \sum_{n\geq m} \sum_{\gamma_n } P(T(\gamma_n) \leq n+1 \mid B) P(R(\gamma_n) \mid C).$$

The $T(\gamma_n)$ are identically distributed with the same distribution as $T_n$ from Lemma \ref{lem:tau}, and there are $d^n$ oriented paths of length $n$. So, the above line reduces to
$$P(A \mid B,C) \geq 1 - \sum_{n \geq m} d^n P(T_n \leq n+1 \mid B) P(R(\gamma_n) \mid C).$$
It then follows from the bound in Lemma \ref{lem:tau} that for any $\epsilon>0$ and $m$ sufficiently large we have
\begin{align}P(A \mid B,C) \geq 1 - \sum_{n \geq m} d^n (p+ \epsilon)^n P(R_x \mid C).\label{eq:A}\end{align}

 A simple way to bound $P(R(\gamma_n)\mid C)$ is to observe that $R(\gamma_n)$ can only occur if each vertex of $\gamma_n$ has at least one incident edge with red passage time $1$. This occurs with probability $1-(1-p)^d$ at each of the $n$ vertices of $\gamma_n$. Let $c = 1/P(C)$. Using the trivial bound $P(R(\gamma_n) \cap C) \leq P(R(\gamma_n))$ and the formula for conditional probability, it follows that for $\|x\|_1 = n$
\begin{align*}
P(R(\gamma_n) \mid C) \leq  \f{P(R(\gamma_n))}{P(C)} \leq c(1- (1-p)^d)^n .
\end{align*}
Our choice of $d$ at \eqref{eq:d} guarantees that $1 - (1-p)^d < 1- 3d^{-1}$.  Thus, 
\begin{align}
P(R(\gamma_n) \mid C) \leq c (1- 3d^{-1})^n \label{eq:R}.	
\end{align}

Applying \eqref{eq:R} to \eqref{eq:A} and recalling that $p= d^{-1} + d^{-3}$ gives 
$$P(A \mid B,C) \geq 1- c\sum_{n\geq m} [d (d^{-1} + d^{-3}+\epsilon) ( 1- 3d^{-1})]^n.$$
Our choice of $\epsilon$ at \eqref{eq:epsilon} yields
$$P(A \mid B,C) \geq 1- c \sum_{n \geq m} (1- \delta)^n.$$
This lower bound converge to $1$ as $m\to \infty$, which completes the proof.
\end{proof}

\bibliographystyle{amsalpha}
\bibliography{chase_escape}

\providecommand{\bysame}{\leavevmode\hbox to3em{\hrulefill}\thinspace}
\providecommand{\MR}{\relax\ifhmode\unskip\space\fi MR }
\providecommand{\MRhref}[2]{%
  \href{http://www.ams.org/mathscinet-getitem?mr=#1}{#2}
}
\providecommand{\href}[2]{#2}
\begin{thebibliography}{ADH17}

\bibitem[ABK90]{aldous}
David Aldous and William B.~Krebs, \emph{The birth-and-assassination process},
  Statistics \& Probability Letters \textbf{10} (1990), 427--430.

\bibitem[ADH17]{fpp}
Antonio Auffinger, Michael Damron, and Jack Hanson, \emph{50 years of
  first-passage percolation}, vol.~68, American Mathematical Soc., 2017.

\bibitem[Ahl15]{ahlberg2015asymptotics}
Daniel Ahlberg, \emph{Asymptotics of first-passage percolation on
  one-dimensional graphs}, Advances in Applied Probability \textbf{47} (2015),
  no.~1, 182--209.

\bibitem[Bor08]{rumor}
Charles Bordenave, \emph{On the birth-and-assassination process, with an
  application to scotching a rumor in a network}, Electron. J. Probab.
  \textbf{13} (2008), 2014--2030.

\bibitem[Bor14]{bordenave2014extinction}
\bysame, \emph{Extinction probability and total progeny of predator-prey
  dynamics on infinite trees}, Electron. J. Probab \textbf{19} (2014), no.~20,
  1--33.

\bibitem[CD81]{cox1981}
J.~Theodore Cox and Richard Durrett, \emph{Some limit theorems for percolation
  processes with necessary and sufficient conditions}, Ann. Probab. \textbf{9}
  (1981), no.~4, 583--603.

\bibitem[CD83]{oriented}
J~Theodore Cox and Richard Durrett, \emph{Oriented percolation in dimensions
  $d\geq 4$: bounds and asymptotic formulas}, Mathematical Proceedings of the
  Cambridge Philosophical Society, vol.~93, Cambridge University Press, 1983,
  pp.~151--162.

\bibitem[Dur84]{durrett_oriented}
Richard Durrett, \emph{Oriented percolation in two dimensions}, The Annals of
  Probability (1984), 999--1040.

\bibitem[GM05]{bernoulli1}
Olivier Garet and R\'egine Marchand, \emph{Coexistence in two-type
  first-passage percolation models}, Ann. Appl. Probab. \textbf{15} (2005),
  no.~1A, 298--330.

\bibitem[GM06]{bernoulli2}
\bysame, \emph{Competition between growths governed by {B}ernoulli
  {P}ercolation}, Markov Processes and Related Fields (2006), 695--734.

\bibitem[Hof08]{hoffman}
Christopher Hoffman, \emph{Geodesics in first passage percolation}, The Annals
  of Applied Probability (2008), 1944--1969.

\bibitem[HP98]{pemantle}
Olle H{\"a}ggstr{\"o}m and Robin Pemantle, \emph{First passage percolation and
  a model for competing spatial growth}, Journal of Applied Probability
  \textbf{35} (1998), no.~3, 683--692.

\bibitem[Kes86]{Aspects}
Harry Kesten, \emph{Aspects of first passage percolation}, \'{E}cole d'\'et\'e
  de probabilit\'es de {S}aint-{F}lour, {XIV}---1984, Lecture Notes in Math.,
  vol. 1180, Springer, Berlin, 1986, pp.~125--264. \MR{876084 (88h:60201)}

\bibitem[KL06]{2typesZ}
George Kordzakhia and Steven Lalley, \emph{A two-species competition model on
  $\mathbb{Z}^d$}, Stochastic Processes and their Applications \textbf{115}
  (2006), 781--796.

\bibitem[Kor05]{tree1}
George Kordzakhia, \emph{The escape model on a homogeneous tree}, Electron.
  Commun. Probab. \textbf{10} (2005), 113--124.

\bibitem[Kor15]{complete}
Igor Kortchemski, \emph{A predator-prey {SIR} type dynamics on large complete
  graphs with three phase transitions}, Stochastic Processes and their
  Applications \textbf{125} (2015), no.~3, 886 -- 917.

\bibitem[Kor16]{tree_chase}
Igor Kortchemski, \emph{Predator--prey dynamics on infinite trees: A branching
  random walk approach}, Journal of Theoretical Probability \textbf{29} (2016),
  no.~3, 1027--1046.

\bibitem[Ric73]{richardson}
Daniel Richardson, \emph{Random growth in a tessellation}, Mathematical
  Proceedings of the Cambridge Philosophical Society, vol.~74, Cambridge
  University Press, 1973, pp.~515--528.

\bibitem[TKL18]{si}
S.~{Tang}, G.~{Kordzakhia}, and S.~P. {Lalley}, \emph{{Phase Transition for the
  Chase-Escape Model on 2D Lattices}}, ArXiv e-prints: 1807.08387 (2018).

\end{thebibliography}

\end{document}